\documentclass[12pt,a4paper,reqno]{amsart}
\usepackage{preemble, relsize, enumerate}

\date{\today}

\author{Stephen Cantrell}
\address{Department of Mathematics, 
University of Warwick,
Coventry, CV4 7AL, UK}
\email{stephen.cantrell@warwick.ac.uk}

\author{Mark Pollicott}
\address{Department of Mathematics, 
University of Warwick,
Coventry, CV4 7AL, UK}
\email{masdbl@warwick.ac.uk}

\date{\today. 
}

\title{Central limit theorems for Green metrics on hyperbolic groups}

\begin{document}
\maketitle

\begin{abstract}
Suppose we have two finitely supported, admissible, probability measures on a hyperbolic group $\G$. In this article we prove that the corresponding two Green metrics satisfy a counting central limit theorem when we order the elements of $\G$ according to one of the metrics. Our results also apply to various other metrics including length functions associated to Anosov representations and to group actions on hyperbolic metric spaces.
\end{abstract}

\section{Introduction}
Suppose that $\Gamma$ is a non-elementary hyperbolic group (for example, a free group, the fundamental group of negatively curved, closed Riemannian manifold, etc). Recently there has been much interest in understanding the statistical behaviour of various natural real valued functions on $\G$.  There has been particular attention given to limit theorems in which group elements in $\G$ are ordered by their word length with respect to a finite generating set.
More precisely, suppose that $\Psi: \G \to \R$ is a real valued function on $\G$ that counts or represents a quantity of interest (for example a quasi-morphism or the displacement associated to the action of $\G$ on a metric space) and that we have fixed a finite generating set $S$ for $\Gamma$. Write $S_n = \{ g\in\Gamma: |x|_S = n \}$ for the collection of group elements of $S$ word length $n$. Multiple authors have studied central limit theorems for the sequence of  counting distributions
\[
\frac{1}{\#S_n} \#\left\{ g \in S_n : \frac{\Psi(g)- \Lambda n}{\sqrt{n}} \le t \right\} \ \text{ where $t\in \R$, $n\ge 1$.}
\]
See for example \cite{horsham.sharp}, \cite{rivin}, \cite{CalegariFujiwara2010}, \cite{chas.lalley}, \cite{calegari.walker}, \cite{calegari.maher},  \cite{GTT2},  \cite{GTT}, \cite{cantrell.sert},  \cite{cantrell}  amongst many other works. In all of these works the authors exploit the automatic structure of hyperbolic groups: the fact that  hyperbolic groups equipped with a generating set can in some sense be encoded by a finite directed graph. Using this automatic structure the sequence of uniform counting measures on $S_n$ can be modelled by the stationary measure for a Markov chain (or measure of maximal entropy for a subshift of finite type). Using this correspondence, it is possible to use tools from probability theory to study counting problems for word metrics.

This discussion leads us to ask whether similar central limit theorems hold when we order group elements by a different metric (other than a word metric). For example, is it possible to prove a central limit theorem similar to the above but with the word metric replaced with the displacement function associated to the action of $\G$ on a hyperbolic metric space or with a Green metric associated to a random walk? 
Not much appears to be known about this problem in general. This is likely because, unlike for word metrics, there is not a well understood probabilistic model that emulates counting over spheres for a general  metric. Consequently, it is often the case that the asymptotic properties of a particular quantity is known when counting with respect to word metrics but not when counting over other metrics. This is the case for stable commutator length for example. Indeed, the large deviation properties of stable commutator length are well-understood when group elements are ordered by a word length. This is due to  Calegari-Walker \cite{calegari.walker}  and Calegari-Maher \cite{calegari.maher}. However the analogous results when ordering by other metrics are unknown, see Conjecture A.11 in \cite{calegari.walker}. 
We discuss these points further in Section \ref{sec.method}.


In specific settings some statistical results with counting over non word metrics are known: \cite{cantrell.pollicott}, \cite{schwartz.sharp}, \cite{lalley}, \cite{babillot.ledrappier}. In these works tools from thermodynamic formalism were used and in particular, delicate Dolgopyat type estimates for norms of certain  transfer operators were exploited. These estimates are known to hold in only a few particular situations.


The aim of this work is to obtain general counting central limit theorems where we count over metrics that are not word metrics and for which we do not have access to Dolgopyat type estimates for corresponding transfer operators. 
We prove a central limit theorem that compares pairs of \textit{strongly hyperbolic} metrics on $\G$. These are metrics that are `stable at infinity': their Busseman functions are H\"older continuous cocycles (see Section $2$ for the precise definition). 
Before stating our results we introduce notation and recall some preliminary facts.
Let $\mathcal{N}_\sigma : \R \to \R$ denote the normal distribution with mean $0$ and variance $\sigma^2$:
\[
\mathcal{N}_\sigma(t) = \frac{1}{\sqrt{2\pi}\sigma} \int_{-\infty}^t e^{-u^2/2\sigma^2} \ du.
\]
 Two metrics $d,d_\ast$ on $\G$ are called \emph{roughly similar} if there is $\Lambda, C >0$ such that $|d(g,h) - \Lambda d_\ast(g,h)| \le C$ for all $x,y \in \G$.  The translation length function of $d$ is defined as the limit
\[
\ell_d[g] = \lim_{n\to\infty} \frac{d(o,g^n)}{n}
\]
and is defined on conjugacy classes $[g]$ in $\G$. Here we use $o\in\G$ to denote the identity element. We call $\G$ \textit{non-elementary} if it does not contain a finite index cyclic subgroup and say that $\G$ is \textit{virtually free} if it 
contains a free group as a finite index subgroup. We say that $d$ has \emph{non-arithmetic length spectrum} if $\ell_d$ does not take values lying in a lattice $a \mathbb{Z}$ for some $a \in\R$. When $\G$ is not virtually free, every strongly hyperbolic on $\G$ has non-arithmetic length spectrum \cite{GMM2018} (also see \cite[Theorem 1.4]{cantrell}).
We let 
\[
\tau(d_\ast/d) := \lim_{T\to\infty} \frac{1}{\#\{g\in\G: d(o,g) <T\}} \sum_{d(o,g) < T} \frac{d_\ast(o,g)}{T} > 0
\]
denote the \textit{mean distortion} of $d_\ast$ with respect to $d$ which exists by \cite[Theorem 1]{cantrell.tanaka.1}.

\begin{theorem}\label{thm.var}
Let $\G$ be a non-elementary hyperbolic group and take two strongly hyperbolic metrics $d,d_\ast \in \Dc_\G$. Suppose either that $\G$ is not virtually free or that $\G$ is virtually free and that $d$ has non-arithmetic length spectrum. 
Let $\tau = \tau(d_\ast/d)$ be the mean distortion.
Then the following limit exists
\[
\sigma^2(d_\ast/d) := \lim_{T\to\infty} \frac{1}{\#\{g \in \G: d(o,g) < T\}} \sum_{d(o,g) < T} \frac{(d_\ast(o,g) - \tau T)^2}{T} \ge 0.
\]
Furthermore the following are equivalent:
\begin{enumerate}
\item $\sigma^2(d_\ast/d) = 0$; and,
\item $d$ and $d_\ast$ are roughly similar.
\end{enumerate}
\end{theorem}
\begin{remark}
We identify $\sigma^2(d_\ast/d)$ with the second derivative at $0$ of the \textit{Manhattan curve} for $d_\ast, d$. See Section \ref{sec.rigid}.
\end{remark}
The constant $\sigma^2(d_\ast/d)$ then appears as the variance in the following central limit theorem.
\begin{theorem}\label{thm.main}
Suppose that $\G$ is a non-elementary hyperbolic group and take two strongly hyperbolic metrics $d,d_\ast \in \Dc_\G$. Suppose either that $\G$ is not virtually free or that $\G$ is virtually free and that $d$ has non-arithmetic length spectrum. Then there exist $\tau = \tau(d_\ast/d) >0$ and $\sigma^2 = \sigma^2(d_\ast/d) \ge 0$ such that
\[
\frac{1}{\#\{ g \in \G: d(o,g) < T\}} \#\left\{ g \in \G: d(o,g) < T \ \text{ and } \ \frac{d_\ast(o,g) - \tau T}{\sqrt{T}} \le t  \right\} \to \mathcal{N}_\sigma(t)
\]
as $T\to\infty$. 
Furthermore, $\sigma^2(d_\ast/d) = 0$ if and only if $d$ and $d_\ast$ are roughly similar.
\end{theorem}
\noindent We stress that this result is new even in the case when $\G$ is a free group or surface group. 

In the rest of this introduction we present some applications of this theorem and then discuss our method of proof. Further applications and results are included in the final section.

\subsection{Random walks}
Suppose that $\mu$ is a finitely supported probability measure, with symmetric support (i.e. $\mu(g) = \mu(g^{-1})$ for all $g \in\G$), on $\G$ such that the support of $\mu$ generates $\G$ as a semi-group. Let $G(\cdot, \cdot)$ denote the corresponding \emph{Green function} defined by 
\[
G(g,h)  = \sum_{g \in\G} \mu^{\ast n} (g^{-1}h) \ \text{ for $g,h \in \G$}
\]
where $\mu^{\ast n}$ denotes the $n$th convolution of $\mu$.
The \emph{Green metric} is defined as
\[
d_\mu(g,h) = - \log \left(\frac{G(g,h)}{G(o,o)}\right)
\]
and intuitively $d_\mu(g,h)$ represents the (minus logarithm of the) probability that a random walk starting at $g$ reaches $h$. Then $d_\mu$ is a strongly hyperbolic metric in $\Dc_\G$ \cite{Nica.Spakula} with $v_{d_\mu} =1$. Hence we obtain the following result.

\begin{theorem} [CLT for Green metrics]
Suppose that $\G$ is a non-elementary hyperbolic group and take two Green metrics $d, d_\ast$ on $\G$ associated to finite, admissible probability measures $\mu, \mu_\ast$ both with symmetric support. Suppose either that $\G$ is not virtually free or that $\G$ is virtually free and that $d$ has non-arithmetic length spectrum. Then there exists $\tau = \tau(d_\ast/d) \ge 1$ and $\sigma^2 = \sigma^2(d_\ast/d) \ge 0$ such that
\[
\frac{1}{\#\{ g \in \G: d(o,g) < T\}} \#\left\{ g \in \G: d(o,g) < T \ \text{ and } \ \frac{d_\ast(o,g) - \tau T}{\sqrt{T}} \le t  \right\} \to \mathcal{N}_\sigma(t)
\]
as $T\to\infty$. Furthermore the following are equivalent
\begin{enumerate}
\item $d$ and $d_\ast$ are roughly similar;
\item $\sigma^2  = 0$; and,
\item $\tau = 1$.
\end{enumerate}
\end{theorem}
In fact, it is even possible to weaken the assumption that $\mu, \mu_\ast$ have finite support to having super-exponential moment by a result of Gou\"ezel \cite{GouezelAncona}. Furthermore, in the case when $\G$ is a free group or surface group we can also remove the assumption that $\mu, \mu_\ast$ are symmetric \cite{GouezelLocalLimit}.

\begin{remark}
Our theorem provides an equivalent formulation of the Singularity Conjecture.
Suppose that $\G$ is a surface group and that $\mu$ is a finitely supported admissible probability measure on $\G$. Let $d_\mu$ be the corresponding (not necessarily symmetric) Green metric and write $d$ for (the lift of) a hyperbolic metric on $\G$ coming from the action of $\G$ on $\mathbb{H}$. The  Singularity Conjecture predicts that for any $\mu$ as above, the corresponding hitting measure
on $\partial \G = S^1$ is mutually singular with respect to the Lebesgue measure.  The conjecture is equivalent to the assertion that for $\mu$ as above, $d_\mu$ and $d$ are never roughly similar \cite{GMM2018}. Despite recent progress \cite{lessa.1}, \cite{lessa.2}, \cite{kosenko}, \cite{kosenko.tiozzo}, the problem remains open. Theorem \ref{thm.main}  provides the following equivalent formulation of the singularity conjecture: the metrics $d_\mu$ and $d$ satisfy a non-degenerate central limit theorem (i.e. $\sigma^2(d_\mu/d) > 0$) as in Theorem \ref{thm.main} if and only if the hitting measure of $\mu$ is mutually singular with respect to the Lebesgue measure. This result also generalises to higher dimensions.
\end{remark}

\subsection{Anosov representations}
Recall that a representation $\rho: \G \to \SL_d(\R)$ (of a finitely generated group $\G$) is projective Anosov if the first and second singular values of $\rho$ separate exponentially as word length increases, i.e. given a finite generating set $S$ on $\G$ there exists $C, \Lambda >0$ such that
\[
\frac{\sigma_1(\rho(g))}{\sigma_2(\rho(g))} \ge C e^{\Lambda |g|_S}
\]
for all $g\in\G$. Here given $A \in \SL_d(\R)$, $\sigma_1(A) \ge \sigma_2(A) \ge \ldots \sigma_d(A)$ are the singular values of $A$. 
This is not the original definition of projective Anosov introduced by Labourie \cite{Labourie} (for surface groups) and extended (to all groups) by Guichard and Wienhard \cite{GuichardWeinhard} however it is easy to state and is known to be equivalent by the work \cite{BochiPotrieSambarino}.  Furthermore, groups admitting a projective Anosov representation are necessarily hyperbolic \cite{BochiPotrieSambarino}.
 The Hilbert length functional associated to $\rho: \G \to \SL_d(\R)$ is given by
\[
\alpha(g,h) := \log\sigma_1(\rho(g^{-1}h)) - \log\sigma_d(\rho(g^{-1}h)) \ \text{ for $g,h \in \G$.  }
\]
These are examples of strongly hyperbolic metrics (this follows from \cite[Lemma 7.1]{cantrell.tanaka.2}). We therefore deduce the following.
\begin{theorem}[CLT for Anosov representations]
Suppose that $\G$ is non-elementary, is not virtually free and that $\alpha, \alpha_\ast: \G \to \R$ are Hilbert length functions associated to two projective Anosov representations $\rho : \G \to \SL_n(\R), \rho_\ast : \G \to \SL_m(\R)$ as above. Then there exists $\tau = \tau(\rho_\ast/\rho) >0$ and $\sigma^2 = \sigma^2(\rho_\ast/\rho) \ge 0$ such that
\[
\frac{1}{\#\{ g\in\G: \alpha(o,g) < T\}} \#\left\{ g\in\G: \alpha(o,g) < T \ \text{ and } \ \frac{\alpha_\ast(o,g) - \tau T}{\sqrt{T}} \le t  \right\} \to \mathcal{N}_\sigma(t)
\]
as $T\to\infty$. 
\end{theorem}
For Hitchin representations, the variance being $0$ in the above theorem determines the representation entirely (i.e.  $\sigma = 0$ if and only if the representations are the same or are contragredient). This also holds for Benoist representations, see \cite[Remark 7.11]{weinhardetal}.

\subsection{Rigidity results}\label{sec.rigid}
When comparing metrics $d,d_\ast$ a natural object to study is their \textit{Manhattan curve} $\theta_{d_\ast/d}$. This curve is defined as follows, for each $s \in \R$ we define $\theta_{d_\ast/d}(s)$ to be the abscissa of convergence of 
\[
\sum_{x\in\G} e^{-sd_\ast(o,x) - td(o,x)}
\]
as $t$ varies.  In \cite{cantrell.tanaka.2} it was shown that for pairs of strongly hyperbolic metrics $\theta_{d_\ast/d}$ is analytic. It was also shown that the constant $- \tau(d_\ast/d)$ is the derivative at $0$ of $\theta_{d_\ast/d}(s)$, i.e. $\theta'_{d_\ast/d}(0) = - \tau(d_\ast/d)$. 
By studying the Manhattan curve the first author  and Tanaka proved various rigidity results for pairs of metrics in $\Dc_\G$. In particular they show that the following are equivalent (see Theorem \ref{thm.equiv} below):
\begin{enumerate}
\item $d$ and $d_\ast$ are roughly similar;
\item the Manhattan curve $\theta_{d_\ast/d}$ is a straight line;
\item $\tau(d_\ast/d) = v_d/v_{d_\ast}$; and,
\item there exists $\tau > 0$ such that $\ell_d[x] = \tau \ell_{d_\ast}[x]$ for all $x \in \G$.
\end{enumerate}
Here $v_d$ (respectively $v_{d_\ast}$) is exponential growth rate of $d$ (respectively $d_\ast$), i.e.
\[
v_d = \limsup_{T\to\infty} \frac{1}{T} \log \#\{ x \in \G: d(o,x) < T\} \ \ \text{ and } \ \ v_{d_\ast} = \limsup_{T\to\infty} \frac{1}{T} \log \#\{ x \in \G: d_\ast(o,x) < T\}.
\]
When any of the above statements fail to hold it was shown that $\theta_{d_\ast/d}$ is a globally strictly convex function and in fact $\theta_{d_\ast/d}'' (t) > 0$ for all $t \in \R$. A large deviation principle comparing $d,d_\ast$ was also shown in \cite{cantrell.tanaka.1}.

Theorem \ref{thm.var} and Theorem \ref{thm.main} then provide an additional rigidity characterisation to those above. As with the constant $\tau(d_\ast/d)$ we can see $\sigma^2(d_\ast/d)$ in the Manhattan curve $\theta_{d_\ast/d}$ and in this case as a second derivative.

\begin{theorem}\label{thm.var2}
The constant $\sigma^2(d_\ast/d) \ge 0$ appearing in Theorem \ref{thm.var} can be realised as $\sigma^2(d_\ast/d) = \theta_{d_\ast/d}''(0)$.
\end{theorem}
We prove this result in Section \ref{sec.proofs} along with Theorem \ref{thm.var}.
\begin{remark}
This result has the following consequence.
When $\G$ is the fundamental group of a closed hyperbolic surface of genus $\mathfrak{g}$ and $\rho_t$ for $-\epsilon < t < \epsilon$ is a smooth (non-constant) family of Riemannian surfaces then the Weil-Petersson metric can be recovered as
\[
\left\|\left.\frac{d}{dt}\right|_{t=0} \rho_t\right\|^2_{\text{WP}} = \frac{3\pi(\mathfrak{g}-1)}{2}\left.\frac{d^2}{dt^2}\right|_{t=0} \sigma^2(d_t/d_0)
\]
where each $d_t$ for $-\epsilon < t < \epsilon$ is the lift of $\rho_t$ to $\G$. This follows from Theorem \ref{thm.var} and \cite[Lemma 4.2]{sharp.pollicott.wp} and complements results of McMullen \cite{mcmullen}.
\end{remark}

\subsection{Ideas behind the proof and structure of the paper}\label{sec.method}
We now outline our method of proof. We note that our theorems do not readily follow from the works of the second author and Tanaka \cite{cantrell.tanaka.2}, \cite{cantrell.tanaka.1} (in which the corresponding large deviation principle was proven). This is because, although the methods and techniques introduced in those papers can be used to study Manhattan curves, they cannot be used to study versions of the Manhattan curve with complex perturbations/parameters. Understanding complex versions of Manhattan curves is vital as they are linked to the Fourier transforms of our counting distributions (and hence they determine whether the central limit theorem holds).  We also note that it does not appear possible to tackle Theorem \ref{thm.main} using random walks on groups as in \cite{GTT} or Markov chains as in \cite{cantrell.sert}. This is because, as we discussed above, unlike for counting over group elements of word length at most $T$ as in \cite{GTT} and \cite{cantrell.sert},  there is not a well understood probabilistic model for counting over the group elements $g$ satisfying $d(o,g) < T$.

To prove Theorem \ref{thm.main} we use the method of moments and techniques from thermodynamic formalism. That is, we show that for each $p\ge 1$ the $p$th moments of the sequence of distributions in Theorem \ref{thm.main} converge as $T\to\infty$ to the appropriate limit (see equation (\ref{eq.moments})). To do this we
\begin{enumerate}
\item introduce a two variable Poincar\'e series $\eta$ that encodes information about $d$ and $d_\ast$;
\item prove results about the domain of analyticity and poles of $\eta$ (this step uses Cannon's automatic structure for hyperbolic groups and builds upon recent results from \cite{cantrell.tanaka.2}, \cite{cantrell}); and,
\item extract the asymptotics we need for the method of moments using a Tauberian Theorem. 
\end{enumerate}
This method of proof was inspired by the work of Morris \cite{Morris} on the statistical properties of the  Euclidean algorithm and also the work of Hwang and Janson \cite{hwang.janson}.

The idea of proving asymptotic counting  results through  Poincar\'e series has been used by many authors. The novelty of our current work is overcoming difficult technical issues that appear in (2) above. During this step we use techniques from thermodynamic formalism which we introduce through the use of Cannon's coding: a finite direct graph that represents our hyperbolic group equipped with a finite generating set (see Definition \ref{def.sms}). Our first technical issue stems from the fact that it is not known whether it is alway possible to find a connected Cannon coding. This has been a recurring issue in the ergodic theory of hyperbolic/negatively curved groups. When studying counting results for word metrics, a beautiful argument of Calegari and Fujiwara \cite{CalegariFujiwara2010} circumvents this issue by showing that  there is consistent behaviour between different connected components in the Cannon graph (specifically the components that are in some sense `big' for the word metric). This argument (and its variants) have  been used by many authors \cite{GTT}, \cite{stats}, \cite{cantrell.sert}, \cite{GTT}. Notably, this argument was used by Gouez\"el in \cite{GouezelLocalLimit} to study the asymptotic growth rate of the transition probabilities for random walks on hyperbolic groups.

 Unfortunately the Calegari - Fujiwara argument and its variants cannot be used in our current setting. This is because these arguments are not suitable for dealing with the large complex perturbations that appear in the study of our Poincar\'e series in (1). To overcome this issue we combine and then build upon arguments from \cite{cantrell.tanaka.2}  and \cite{cantrell}. First we use a result from \cite{cantrell.tanaka.2} that shows that the Mineyev topological flow (a coarse geometric analogue of a geodesic flow) can be coded by the suspension flow over a subshift of finite type coming from the Cannon coding. We then combine this result with an argument from \cite{cantrell} to prove a precise complex version of Calegari-Fujiwara's argument that compares the components in the Cannon graph that are `big' for our counting metric $d$. This argument culminates in Lemma \ref{lem.pressagree} which is the crucial technical result that allows us to carry out our analysis. We hope that this result will have other applications in the ergodic theoretic study of hyperbolic groups. 


After obtaining Lemma \ref{lem.pressagree} there is one final additional technical difficulty that we need to overcome: in our general setting we do not have access to strong Dolgopyat type bounds for appropriate transfer operators corresponding to our counting metric. To get around this we show that the assumption that $\G$ is not virtually free (or $d$ has non-arithmetic length spectrum) implies strong enough spectral properties for our transfer operators (Proposition \ref{prop.srless}) and that these are enough to deduce the central limit theorem. 
\subsection*{Organisation}
The paper is structured as follows. In Section $2$ we cover preliminary material related to hyperbolic groups and geometries. We also, in this section, state and prove the technical domain of analyticity result mentioned in $(2)$ above. The proof of this result (which appears as Proposition \ref{prop.analytic}) will occupy most of the second section. In Section $3$ we use the method of moments to deduce Theorem \ref{thm.main} and in the final section, we present some further applications of our results.


\section{Preliminaries}

\subsection{Hyperbolic groups and metrics}

We only briefly introduce the required preliminary materials concerning hyperbolic groups and metrics here. See Section 2 of \cite{cantrell.tanaka.1} for a more detailed account.

\begin{definition}
A metric space $(X, d)$ is said to be $\d$-{\it hyperbolic} for some $\delta \ge 0$ if
\begin{equation*}
(x,y)_w \ge \min\left\{(x,z)_w, (y,z)_w\right\}-\d \quad \text{for all $x, y, z, w \in X$}
\end{equation*}
where $(x,y)_w$ is the Gromov product:
\[
 (x,y)_w =\frac{1}{2}( d(w, x)+d(w, y)-d(x, y)) \ \ \text{for $x, y, w \in X$}.
\]
A metric space is called {\it hyperbolic} if it is $\d$-hyperbolic for some $\d \ge 0$.
\end{definition}

A \emph{hyperbolic group} is a finitely generated group $\G$ such that the Cayley graph $(\Cay(\G,S), |\cdot|_S)$ equipped with the word metric associated to a finite generating set $S$ is a hyperbolic metric space. All hyperbolic groups will be assumed to be \textit{non-elementary}, i.e. they do not contain a finite index cyclic subgroup. In the statement of 
Theorem \ref{thm.var} and 
Theorem \ref{thm.main} the assumption of $\G$ not being \emph{virtually free} means that $\G$ does not contain a free group as a finite index subgroup.

\begin{definition}
We say that two metrics $d,d_\ast$ on $\G$ are {\it quasi-isometric} if there exist constants $L > 0$ and $C \ge 0$ such that
\[
L^{-1} \, d(x, y)-C \le d_\ast(x, y) \le L \, d(x, y)+C \ \text{ for all $x, y \in \G$}.
\]
\end{definition}
Throughout this work,  $\Dc_\G$ will denote the set of metrics which are left-invariant, hyperbolic and quasi-isometric to some (equivalently, any) word metric in $\G$. 
Recall that for $d\in \Dc_\G$ we use the notation $v_d$ for the exponential growth rate of $d$, i.e.
\[
v_d = \lim_{T\to\infty} \frac{1}{T} \log \#\{ g\in\G: d(o,g) < T\} \ \text{(which is necessarily strictly positive)}.
\]
Furthermore, for any $d \in \Dc_\G$ there exists $C_1, C_2 > 0$ such that
\[
C_1 e^{v_d T} \le \#\{x \in \G: d(o,x) < T\} \le C_2 e^{v_d T} \ \text{ for all $T > 0$.}
\]
We recall that two metrics $d,d_\ast \in \Dc_\G$ have the property that there exists $\tau, C >0$ such that $|\tau d(o,g) - d_\ast(o,g)| < C$ for all $g\in\G$ then we say that $d,d_\ast$ are \textit{roughly similar}. We recall that $o\in\G$ is the identity element. 
The translation length function of $d \in \Dc_\G$ is defined as the limit
\[
\ell_d[g] = \lim_{n\to\infty} \frac{d(o,g^n)}{n}
\]
and is defined on conjugacy classes $[g]$ in $\G$. The \textit{length spectrum} of $d$ is the collection $\{ \ell_d[g] : g \in \G\}$ and we say that $d$ has \textit{arithmetic length spectrum} if this set is contained in a lattice $a \Z$ for some $a \in \R$. Otherwise we say that $d$ has \textit{non-arithmetic length spectrum}.

Fix a metric $d$ in $\Dc_\G$.
\begin{definition}
Given an interval $I \subset \R$ and constants $L, C >0$ we say that a map $\g: I \to \G$ is an \textit{$(L, C)$-quasi-geodesic} if 
\[
L^{-1} \, |s-t|-C \le d(\g(s), \g(t)) \le L\, |s-t|+C \  \text{ for all $s, t \in I$},
\]
and a \textit{$C$-rough geodesic} if
\[
|s-t|-C \le d(\gamma(s), \gamma(t))\le |s-t|+C \ \text{ for all $s, t \in I$}.
\]
\end{definition}
A $0$-rough geodesic is referred to as a geodesic.
A metric space $(\G, d)$ is called $C$-{\it roughly geodesic} if for each pair of elements $x, y \in \G$ we can find a $C$-rough geodesic joining $x$ to $y$. We say that $(\G,d)$ is \textit{roughly geodesic} if it is $C$-roughly geodesic for some $C \ge 0$.
It is known that every metric in $\Dc_\G$ is roughly geodesic \cite{BonkSchramm}.

Our main results are concerning the following class of metrics.

\begin{definition}
A hyperbolic metric $d$ in $\G$ is called {\it strongly hyperbolic} if there exist positive constants $c, R_0 > 0$ such that for all $R \ge R_0$, and all $x, x', y, y' \in \G$,
if
$ d(x, y)-d(x, x')+d(x', y')-d(y, y') \ge R,$
then
\[
|d(x, y)-d(x', y)-d(x, y')+d(x', y')| \le e^{-c R}.
\]
\end{definition}

\begin{remark} 
Examples of strongly hyperbolic metrics include:

\noindent (1) The Green metric associated to random walk on $\G$ as discussed in the introduction.

\noindent (2) The Mineyev hat metric introduced in \cite{MineyevFlow}.

\noindent (3) The orbit metric coming from a properly discontinuous, cocompact, free and isometric action on a $\text{CAT}(-1)$ metric space.

\noindent (4) Linear functions on the Cartan algebra associated to Anosov representations.
\end{remark}

Hyperbolic groups can be compactified using their Gromov boundary $\partial \G$ which consists of equivalence classes
of divergent sequences. We fix a reference metric $d \in \Dc_\G$ and say that sequence of group elements $\{g_n\}_{n=0}^\infty$ diverges
if $(g_n|g_m)_o$ (computed with respect to $d$) diverges as $\min\{n, m\}$ tends to infinity. Two divergent sequences $\{g_n\}_{n=0}^\infty$ and $\{h_n\}_{n=0}^\infty$ are {\it equivalent} if $(g_n|h_m)_o$ diverges as $\min\{n, m\}$ tends to infinity. 
If $d \in \Dc_\G$ is $C$-roughly geodesic then for each $\x$ in $\partial \G$ there exists a $C$-rough geodesic $\g:[0, \infty) \to \G$ 
such that $\g(0)=o$ and $\g(n) \to \x$ as $n \to \infty$.

The Busemann function $ \beta_w(g, \x)$ associated to $d \in \Dc_\G$ based at $w \in \G$ is given by
\[
\beta_w(g,\x) = \sup\left\{ \limsup_{n\to\infty} d(g,\x_n) - d(w,\xi_n) : \x = \{ \x_n\}_{n=0}^\infty \right\}.
\]
When $d$ is strongly hyperbolic $\beta_w$ is obtained as a limit 
\[
\b_w(g, \x)=\lim_{n \to \infty}\( d(g, x_n)- d(w, x_n)\) \quad \text{for $(g, \x) \in \G \times \partial \G$},
\]
where $\x=\{g_n\}_{n=0}^\infty$,
and is continuous with respect to $\x$ in $\partial \G$.
Moreover in this case we have the cocycle identity
\[
\b_w(gh, \x)= \b_w(h, g^{-1}\x)+ \b_w(g w, \x) \quad \text{for $w, g, h \in \G$ and $\g\in\G\cup \partial \G$}.
\]

Before moving on we recall the following results of the first author  and Tanaka.
Recall that the \textit{Manhattan curve} $\theta_{d_\ast/d}$ is defined as follows. For each $s \in \R$ define $\theta_{d_\ast/d}(s)$ to be the abscissa of convergence of 
\[
\sum_{x\in\G} e^{-sd_\ast(o,x) - td(o,x)}
\]
as $t$ varies.  In \cite{cantrell.tanaka.1} and \cite{cantrell.tanaka.2} it was shown that for pairs of strongly hyperbolic metrics $d,d_\ast$, $\theta_{d_\ast/d}$ is analytic. The following was also shown.
\begin{theorem}[Theorem 1 of \cite{cantrell.tanaka.1}]\label{thm.equiv}
Let $d,d_\ast \in \Dc_\G$ be metrics. Then the limit
\[
\tau(d_\ast/d) := \lim_{T\to\infty} \frac{1}{\#\{g\in\G: d(o,g) <T\}} \sum_{d(o,g) < T} \frac{d_\ast(o,g)}{T} > 0
\]
exists and is called the \emph{mean distortion} of $d_\ast$ with respect to $d$. We also have that $\tau(d_\ast/d) = - \theta_{d_\ast/d}'(0)$ and $\tau(d_\ast/d) \ge v_d/v_{d_\ast}$. Further, the following are equivalent:
\begin{enumerate}
\item $d$ and $d_\ast$ are roughly similar;
\item $\theta_{d_\ast/d}$ is a straight line;
\item $\tau(d_\ast/d) = v_d/v_{d_\ast}$; and,
\item there exists $\tau > 0$ such that $\ell_d[x] = \tau \ell_{d_\ast}[x]$ for all $x \in \G$.
\end{enumerate}
\end{theorem}

We also recall one of the main results from \cite{cantrell}: for any strongly hyperbolic metric $d \in \Dc_\G$ with non-arithmetic length spectrum there exists $C > 0$ such that
\[
 e^{- v_d T} \, \#\{ g \in \G: d(o,g) < T\} \to C
\]
as $T\to\infty$. This is usually referred to as the orbital counting asymptotic.

\subsection{Poincar\'e series and domains of analyticity}\label{sec.analytic}
Throughout the rest of this work we fix two strongly hyperbolic metrics $d,d_\ast \in \Dc_\G$ and assume that they have been chosen/scaled so that the following conditions hold. 

\medskip
\noindent
\textit{Standing assumptions:} We assume $\G$ is not virtually free or that $\G$ is virtually free and $d$ has non-arithmetic length spectrum. Suppose that $d,d_\ast$ are not roughly similar and have been scaled so that 
\begin{enumerate}
\item $d$ has exponential growth rate $1$; and,
\item $\tau(d_\ast/d) = 1$.
\end{enumerate}
We will assume that these conditions hold unless otherwise stated.

We define the following function.
\begin{definition}
For $s, t \in \mathbb{C}$ we formally define the two variable \emph{Poincar\'e series}
\[
\eta(s,t) = \sum_{g\in\G} e^{-sd(o,g) - t(d_\ast(o,g) - d(o,g))}.
\]
\end{definition}
Using recent results of the first author and Tanaka \cite{cantrell.tanaka.2}, \cite{cantrell} we analyse the domain of analyticity of $\eta$ to prove the following proposition.
\begin{proposition}\label{prop.analytic}
For each $p \ge 0$ the series
\[
\eta_p(s) = \sum_{g\in\G} (d_\ast(o,g) -  d(o,g))^p e^{-sd(o,g)}
\]
is analytic in the region $\mathfrak{Re}(s) \ge 1$ except for a singularity at $s = 1$. 
The following then hold depending on the parity of $p$.
\begin{enumerate}
\item  When $p$ is odd this singularity consists of integer order poles that have orders bounded above by $(p+1)/2$; and
\item when $p$ is even, in a neighbourhood $U$ of $1$, we can write
\[
\sum_{g\in\G} (d_\ast(o,g) - d(o,g))^p e^{-sd(o,g)} = \frac{R_p(s)}{(s-1)^{1 + \frac{p}{2}}}
\]
where $R_p(s)$ is analytic in $U$. Furthermore when this is the case 
\[
R_p(1) = \frac{C p! \sigma^p}{2^{\frac{p}{2}}}
\]
for each (even) $ p \ge 0$ where $C, \sigma^2 >0$ are constants independent of $p$.
\end{enumerate}
\end{proposition}
To prove this result we will use tools from thermodynamic formalism.
Fix a finite generating set $S$ for $\G$.
\begin{definition} \label{def.sms}
Let  $\Ac=(\Gc, w, S)$ be a collection where
\begin{enumerate}
\item $\Gc=(V, E, \ast)$ is a finite directed graph with distinguished vertex $\ast$ which we call the \textit{initial state}; and,
\item $w: E \to S$ is a labelling such that
for a directed edge path $(x_0, x_1, \dots, x_{n})$ (where $(x_i, x_{i+1})$ corresponds to a directed edge)
there is an associated path in the Cayley graph $\Cay(\G, S)$ beginning at the identity: the path corresponds to
\[
(o, w(x_0,x_1), w(x_0,x_1)w(x_1,x_2), \dots, w(x_0,x_1) \, \cdots \, w(x_{n-1}, x_n)).
\]
\end{enumerate}
Let $\ev$ denote the map that sends a finite path to the endpoint of the corresponding path in $\Cay(\G,S)$, 
$\ev(x_0,\ldots, x_n) = w(x_0,x_1) \, \cdots \, w(x_{n-1},x_n)$.
We say that $\Ac$  is a strongly Markov structure for the pair $\G, S$ if 
\begin{enumerate}
\item for each vertex $v\in V$ there exists a directed path from $\ast$ to $v$;
\item for each directed path in $\Gc$ the associated path in $\Cay(\G, S)$ is a geodesic; and,
\item the map $\ev$ defines a bijection between the set of directed paths from $\ast$ in $\Gc$ and $\G$.
\end{enumerate}
\end{definition}

It is well known that every hyperbolic group and finite generating set admits a strongly Markov automatic structure (cf. \cite{GhysdelaHarpe},  \cite[Section 3.2]{Calegari}) which we call a {\it Cannon coding}.

For technical reasons we augment Cannon codings by introducing an additional vertex labelled $0$. We also add directed edges from every vertex $x \in V \cup \{0\} \backslash \{\ast\}$ to $0$ and define $w(x,0) = o$ (the identity in $\Gamma$) for every $x \in V \cup \{0\} \backslash \{\ast\}$. We assume that every Cannon coding has been augmented in this way and will abuse notation by labelling the augmented structure, its edge and vertex set by $\mathcal{G}$, $V$ and $E$ respectively. We use the notation $\dot{0}$ to denote the infinite sequence consisting of only $0$s.

Using  $\mathcal{G}$ we introduce a subshift of finite type as follows. Let $A$ be the $k \times k$ (where $k$ is the cardinality of $V$), $0-1$ transition matrix describing $\Gc$.
We use $A(i,j)$ to denote the $(i,j)$th entry of $A$. The one-sided subshift of finite type associated to $A$ is the space of all infinite one-sided sequences allowed by $A$:
\[
\Sigma_A = \{(x_n)_{n=0}^{\infty} : x_n \in \{1,2,...,k\}, A(x_n,x_{n+1})=1, n \in \mathbb{Z}_{\ge 0}\}.
\]
When $A$ is clear, we will drop $A$ from the notation in the definition of the shift spaces and simply write $\SS$. As in the definition of $\Sigma_A$, we write $x_n$ for the $n$th coordinate of $x \in \Sigma$. 
The shift map $\sigma: \Sigma_A \rightarrow \Sigma_A$ shifts the entires of a sequence to the left by one and deletes the initial entry, i.e. it sends $x$ to $\sigma(x) = y$ where $y_n=x_{n+1}$ for  $n \in \mathbb{Z}_{\ge 0}$. 
For each $0<\theta <1$ we can define a metric $d_\theta$ on $\Sigma_A$:
given $x,y \in \Sigma_A$ set
 \[
d_\theta(x,y) = 
     \begin{cases}
       \theta^{N} &\quad \text{if $x_0 = y_0$ and $N \ge 1$ is the smallest integer $n$ such that $x_n \neq y_n$} \\
       1 &\quad \text{if $x_0 \neq y_0$}.
     \end{cases}
 \]
 We then define the vector space
\[
F_\theta(\Sigma_A) = \{ r : \Sigma_A \to \mathbb{C}:\text{$r$ is Lipschitz with respect to $d_\theta$} \}
\]
which we equip with the norm $\|r\|_\theta = |r|_\theta + |r|_\infty$ where $|r|_\infty$ is the usual sup-norm and $|r|_\theta$ denotes the least Lipschitz constant for $r$. The space $(F_\theta(\Sigma_A), \|\cdot\|_\theta)$ is a Banach space. Two functions $r_1,r_2 \in F_\theta(\SS_A)$ are said to be cohomologous if there exists continuous $\psi :\Sigma_A \rightarrow \mathbb{C}$ such that $r_1=r_2+\psi \circ \sigma - \psi.$ A well-known theorem of Livsic asserts that for $r\in F_\theta$, $r$ is cohomologous to the constant function with value $C$ if and only if  the set $\{S_nr(x) - Cn: x\in \Sigma_A, n \in \mathbb{Z}_{\ge 0}\}$ is bounded. Here $S_nr(x) = r(x) + r(\sigma(x))+...+r(\sigma^{n-1}(x))$ denotes the $n$th Birkhoff sum evaluated at $x\in\Sigma_A$. We will use this notation for Birkhoff sums throughout this work.

It will be convenient to extend the map $\ev$ defined on finite paths in Definition \ref{def.sms} to infinite paths.
If $x = (x_k)_{k=0}^\infty \in \SS$ and $n \in \Z_{\ge 0}$ then we set
\[
\ev_n(x) = \ev(x_0, x_1, \ldots, x_n) \in \G.
\]
We also define $\ev(x)$ to be the point in $\G \cup \partial \G$ determined by the $S$ geodesic corresponding to $x$, which we think of as starting at the identity in the Cayley graph for $\G,S$.

We say that a directed graph $\Gc$ is connected if there exists a (directed) path between any pair of vertices in $\Gc$. A \textit{connected component} of a finite directed graph is a maximal, connected subgraph.
A Cannon coding $\Gc$ associated to a pair $\G,$ $S$ will never be connected (as the $\ast$ state has only outgoing edges). We can however decompose $\Gc$ into connected components and apply techniques from symbolic dynamics to these components separately.

Given such a connected component $\Cc$, there exists a maximal integer $p_\Cc \ge 1$, known as the period, such that the length of every closed loop in $\Cc$ has length divisible by $p_\Cc$. 
When $p_\Cc =1$ we say that the component $\Cc$ is aperiodic. When $\Cc$ is aperiodic the corresponding subshift  $\SS_\Cc$  is topologically mixing. In general we can only be sure that $\Cc$ is connected  and that $(\SS_\Cc, \sigma)$ is topologically transitive. 
In this case $p_\Cc >1$  and we can decompose the vertex set $V(\Cc)$ for $\Cc$ into $p_\Cc$ disjoint collections of vertices $V_1, \ldots, V_{p_\Cc}$. Letting $\Sigma_j$ for each $j=1,\ldots, p_\Cc$ denote the set of elements in $\SS_\Cc$ that correspond to sequences starting with a vertex in $V_j$, we have that $\sigma(\Sigma_j) = \Sigma_{j+1}$ where $j,j+1$ are taken modulo $p_\Cc$. Each sub-system $(\SS_j, \sigma^{p_\Cc})$ is a mixing subshift of finite type.

Using the same notation as before, let $(\SS_\Cc, \sigma)$ be the shift space defined over a component $\Cc$.
For a function $\Psi$ on $\SS_\Cc$  we want to introduce the \textit{pressure} of $\P$ on $\SS_\Cc$ (or $\Cc$) and will do so via the variational principle.
In the following $\Mcc( \Cc)$ denotes the set of $\sigma$-invariant probability measures on $\SS_\Cc$ and
$h( \lambda)$ denotes the measure theoretical entropy of $(\Sigma_\Cc, \sigma, \lambda)$ (see Section 3 of \cite{ParryPollicott}).
\def\Pr{{\rm P}}
\begin{proposition}[The Variational Principle, Theorem 3.5 \cite{ParryPollicott}]\label{prop.vp}
For $\P \in F_\theta(\Sigma_\Cc)$ the supremum
\[
\Pr_\Cc(\P) =\sup_{\lambda \in \Mcc(\Cc)}\left\{h( \lambda)+\int_{\SS_\Cc}\P\,d\lambda\right\}
\]
is attained by a unique $\sigma$-invariant probability measure $\m_\Cc$ on $\Sigma_\Cc$. This quantity is referred to as the \textit{pressure} of $\P$ over $\Cc$.
\end{proposition}

The pressures can be related to the spectral radii of the following linear operators, known as \textit{transfer operators}. Fix $\P \in F_\theta(\SS_\Cc)$ and define $L_\Cc : F_\theta(\SS_\Cc) \to F_\theta(\SS_\Cc)$ by
\[
L_{\Cc} \o(x) = \sum_{\s(y) = x} e^{\P(y)} \o(y).
\]
Then the spectral radius of this operator is $e^{\Pr_\Cc(\P)}$ and furthermore, this operator has $p_\Cc$ simple maximal eigenvalues: $e^{2\pi i l/p_\Cc}e^{\Pr_\Cc(\P)}$ for $l =1, \ldots, p_\Cc$. The rest of the spectrum is contained in a the disk in  $\C$ centered at $0$ with radius strictly smaller than $e^{\Pr_\Cc(\P)}$ \cite{ParryPollicott}.

When we are considering a Cannon coding with multiple connected components and $\P \in F_\theta(\SS_A)$, we will write
\[
\Pr(\P)=\max_\Cc \Pr_\Cc(\P),
\]
where $\Cc$ runs over all components in $\Gc$.
\begin{definition}
We call a component $\Cc$ {\it maximal} for $\P$ (or \textit{$\P$-maximal}) if $\Pr(\P)=\Pr_\Cc(\P)$.
If $\Cc$ is maximal for the constant function with value $1$ then we say that $\Cc$ is  \textit{word maximal}.
\end{definition}
Note that the word maximal components for a Cannon coding for $\G,S$ are precisely the components
that have spectral radius given by the growth rate of $|\cdot|_S$.
\begin{lemma}[Lemma 4.8 \cite{cantrell.tanaka.1}]\label{lem.holder}
Take a strongly hyperbolic metric $d$ in $\Dc_\G$ with exponential growth rate $ 1$.
Then, for any Cannon coding and associated subshift $\SS$, the function $\P_d : \Sigma \to \Sigma$ defined by $\P_d(x) = \beta_o(\ev_1(x),\ev(x))$ is a H\"older continuous function satisfying the following properties:
\begin{enumerate}
\item $S_n\P_d(x)=\sum_{i=0}^{n-1}\P_d(\s^i(x))= d(o, \ev_n(x))+O(1)$ for all $x \in \SS$;
\item $\Pr(-\P_d) = 0$; and,
\item when $x_0 = \ast$ we have than $S_n\P_d(x) = d(o,\ev_n(x))$.
\end{enumerate}
\end{lemma}
We then have the following important result which was shown in \cite{cantrell} and relies on the work in \cite{cantrell.tanaka.2}.
\begin{proposition}\label{prop.mc}
Suppose that $d \in \Dc_\G$ is strongly hyperbolic (with exponential growth rate $1$) and let $\P_d$ be the potential associated to $d$ from Lemma \ref{lem.holder}. Then the $-\Psi_d$ maximal components are precisely the word maximal components. 
\end{proposition}

We will need to exploit an additional combinatorial property of hyperbolic groups known as \emph{growth quasi-tightness}.   Fix a finite generating set $S$ for $\G$ a group element $\o \in \G$ and a real number $\D \ge 0$. We say that an $S$-geodesic word \textit{$\D$-contains $\o$}  if it contains a subword $\widetilde{\o}$ of the form $\widetilde{\o} = f_1\o f_2$ for some $f_1, f_2 \in \Gamma$ with $|f_1|_S, |f_2|_S \le \D$. Let $Y_{w,\D}$ be the set of group elements $x \in \G$ for which $x$ can be  written as an $S$-geodesic word which does not $\D$-contain $\o$. We then have the following theorem of Arzhantseva and Lysenok.
  
 \begin{theorem} [Theorem 3 of \cite{arzlys}]
Hyperbolic groups are growth quasi-tight. That is, given a hyperbolic group and  any generating set $S$
  there exists $\D_0> 0$ such that for any $\o \in \Gamma$,
  \[
  \lim_{n\to \infty} \frac{\#(Y_{\o,\D_0} \cap S_n)}{\#S_n}=0.
  \]
  Here $S_n = \{ x \in \G: |x|_S = n\}$ are the group elements in $\G$ of $S$ length $n$.
  \end{theorem}
It is crucial that $\D_0$ does not depend on $\o$. As an immediate consequence we deduce the following result.
 
 \begin{lemma} \label{lem.gqt}
 Fix a word maximal component $\Cc$ in a Cannon coding (for a hyperbolic group $\G$ and generating set $S$). Let $\G_\Cc$ denote the collection of group elements that correspond to a path in $\Cc$. That is, $\G_\Cc$ contains the elements in $\G$ that correspond to multiplying the edge labelings along a path in $\Cc$. Then there exists a finite set $F \subset \G$ such that $F \G_\Cc F = \G$.
 \end{lemma}
This was observed in the proof of Lemma 4.6 in \cite{GMM2018},  see Section 4.5 of \cite{cantrell.tanaka.1} for more details.

Now take a pair $d_\ast, d \in \Dc_\G$ of non-roughly similar, strongly hyperbolic metrics satisfying the assumptions introduced at the start of this section.  Fix a generating set $S$ for $\G$ and a Cannon coding as introduced above. Let $\Phi$ be the potential in the coding given by $\Psi_{d_\ast} - \Psi_d$ where $\Psi_d, \Psi_{d_\ast}$ are obtained as in Lemma \ref{lem.holder}. We will write $\Cc_j$, $j=1,\ldots, m$ for the word maximal  
(equivalently $-\Psi_d$ maximal) components.

We define transfer operators for each maximal component $\mathcal{C}_j$. Let $p_j$ denote the period of the component $\Cc_j$. Suppose the entire Cannon coding is described by the matrix $A$. Let, for $j=1, \ldots, m$ the matrix $C_j$ be the matrix with the same dimensions and entries as $A$ except for the fact that each word maximal component that is not $\mathcal{C}_j$ is replaced by the zero matrix. We then define transfer operators $L_{j,s, t}: F_\theta(\Sigma_{C_j}) \to F_\theta(\Sigma_{C_j}) $ by
\[
L_{j,s,t} f(x) = \sum_{\sigma(y) = x, y \neq \dot{0}} e^{-s\Psi_d(y) - t\Phi(y)} f(y)
\]
for $s, t \in \C$. Recall that $\dot{0} \in \SS$ is the sequence consisting only of $0$s. Note that by construction each $C_j$ has a unique maximal component for $- \Psi_d$. The following proposition records some of the important properties of these operators.

\begin{proposition}\label{prop.sr}
For each $j=1,\ldots, m$ the operators $L_{j,s,t}$ satisfy the following.
\begin{enumerate}
\item Let $\chi \in F_\theta(\SS)$ be the indicator function on the one-cyclinder for the initial vertex $\ast$ in the Cannon coding. Then we have that 
\[
\sum_{|g|_S \in M_j(n)} e^{-s d(o,g) - t(d_\ast(o,g) - d(o,g))} = L_{j,s,t}^n \chi(\dot{0})
\]
where for $n\ge 1$,  $M_j(n) \subset \{g\in\G: |g|_S = n\}$ is the collection of group elements in $\G$ whose corresponding path in $\Gc$ starting at $\ast$ does not enter a word maximal component that is not $\Cc_j$.
\item For $s$ with $\mathrm{Re}(s) > 1 = v_d$, the spectral radius of each $L_{j,s,0}$ is strictly smaller than $1$.
\item There exits $\epsilon > 0$ such that for all $|s - 1| < \epsilon, |t| < \epsilon$ the spectrum of each operator $L_{j,s,t}$ is of the following form: $L_{j,s,t}$ has $p_j$ simple maximal eigenvalues of the form $e^{\Pr_j(-s\P_d - t\Phi)} e^{2\pi i l\p_j}$ and the rest of the spectrum is contained in a disk $\{|s| < \rho < 1\}$ for some $\rho$ independent of $|s-1|<\epsilon, |t| < \epsilon$. Here $\Pr_j(-s\P_d - t\Phi)$ is a bi-analytic extension to $|s-1|<\epsilon, |t| < \epsilon$ of the pressure obtained from the variational principle applied to a fixed word maximal component $\Cc_j$. 
\end{enumerate}
\end{proposition}
\begin{proof}
This is essentially the same as \cite[Proposition 4.3]{cantrell} we leave the details to the reader.
\end{proof}
In part $(3)$ above the quanitites $\Pr_j(-s\P_d - t \Psi)$ represent the pressures of the potential $-s\P_d - t \Psi$ over the different maximal components $\Cc_j$.
We then have the following crucial result which can be seen as a multivariable generalisation of Proposition \ref{prop.mc}.
\begin{lemma}\label{lem.pressagree}
There is an open neighbourhood $U \times V \subset \mathbb{C}^2$ of $(1,0)$ such that for\ $(s,t) \in U \times V$ we have that
\[
\textnormal{P}_i(-s\Psi_d - t\Phi) = \textnormal{P}_j( -s\Psi_d - t\Phi)
\]
for each $i,j \in \{1,\ldots, m\}$, i.e. the pressure of $-s\Psi_d - t\Phi$ over the word maximal or equivalently $-\P_d$-maximal components agree.
\end{lemma}

\begin{proof}
We first prove the lemma assuming that $s$ and $t$ are real (i.e. that $U,V$ are sets of real numbers) and then deduce the general case. We have that for each $j = 1,\ldots, m$ we have that $\Pr_j(-s\Psi_d - t\Phi)$ is the exponential growth rate
\[
\Pr_j(-s\Psi_d - t\Phi) = \limsup_{n\to\infty} \frac{1}{n} \log \left(\sum_{\sigma^{n}(x) = x : x \in \Sigma_j} e^{-s\Psi^n_d(x) - t\Phi^n(x)}\right).
\]
 We claim that each of these growth rates are the same as the exponential growth rate
\[
 \limsup_{n\to\infty} \frac{1}{n} \log \left( \sum_{|g|_S = n} e^{-sd(o,g) - t(d_\ast(o,g) - d(o,g) )}\right).
\]
This implies our desired conclusion as this latter quantity is independent of $\mathcal{C}_j$.

To prove the claim we first observe the elementary inequality
\[
\sum_{\sigma^{n}(x) = x : x \in \Sigma_j} e^{-s\Psi^n_d(x) - t\Phi^n(x)} \le C_{s,t} \sum_{|g|_S = n} e^{-sd(o,g) - t(d_\ast(o,g) - d(o,g) )}
\]
which holds for each $n\ge 1$ where $C_{s,t}$ is a constant depending only on $s$ and $t$. This follows from Lemma \ref{lem.holder} and the properties of the Cannon coding.

For the other inequality we use that hyperbolic groups are growth quasi-tight. Fix $j \in \{1,\ldots,m\}$. By Lemma \ref{lem.gqt} there exists a finite subset $F \subset \Gamma$ such that for any $g \in \Gamma$ we can find a loop $\sigma^n(x) = x$ in $\Sigma_j$ such that the group element $\widetilde{g} \in\G$ corresponding to multiplying the $n$ labels in $x$ is equal to $f_1gf_2$ for some $f_1, f_2 \in F$.  Then the association $g \mapsto \widetilde{g}$ is at most $\# F^2$ to $1$. Also we have that $|\widetilde{g}|_S = |g|_S + O(1), d(o,\widetilde{g}) = d(o,g) + O(1), d_\ast(o,\widetilde{g}) = d_\ast(o,g) + O(1)$ where the error terms are independent of $g$. Hence we see that there are $R \in \mathbb{Z}_{\ge 0}$ and $M_{s,t} >0$ depending only on $s$ and $t$ such that
\[
\sum_{|g|_S = n} e^{-sd(o,g) - t(d_\ast(o,g) - d(o,g) )} \le M_{s,t} \sum_{i= -R}^R \sum_{\sigma^{n+i}(x) = x : x \in \Sigma_j} e^{-s\Psi^{n+i}_d(x) - t\Phi^{n+i}(x)}
\]
for each $n \ge 1$. It follows easily from basic results in thermodynamic formalism that the right hand side of the previous inequality has exponential growth rate $\Pr_j(-s\P_d - t\Psi)$. Hence, the two inequalities then conclude the proof.

We now need to deduce the general case. Standard facts from analytic perturbation theory guarantee that there is an open neighbourhood $U \times V \subset \mathbb{C}^2$ of $(1,0)$ such that  the pressures $(s,t) \mapsto \Pr_j(-s\Psi_d - t\Phi)$ are bi-analytic in this neighbourhood. When $(s,t) \in U \times V$ and $s,t\in \mathbb{R}$ we know that the pressure functions $\Pr_j(-s\Psi_d - t\Phi)$ coincide. However since these pressures are bi-analytic on $U \times V$ (and so can be expressed in the form of a $2$ variable Taylor expansion which is calculated from the partial derivatives at $(1,0)$) they are determined by their values on the real values in $(s,t) \in U \times V$ and this concludes the proof.
\end{proof}
This lemma shows that, restricting to a neighbourhood of $(1,0)$,  the pressure functions $\textnormal{P}_i(-s\Psi_d - t\Phi)$ for $i=1,\ldots, m$ coincide. This fact will be vital in the upcoming analysis. Moving forward we will write $P(s,t)$ for this function and define $\lambda(s,t) = e^{P(s,t)}$.
\begin{lemma}\label{lem.deriv}
For $\lambda(s,t)$ as above we have that
\[
\lambda_{s}(1,0) < 0, \  \lambda_{t}(1,0) = 0 \ \text{ and } \ \lambda_{t t}(1,0) > 0.
\]
\end{lemma}

\begin{proof}
Standard facts from thermodynamic formalism \cite{ParryPollicott} imply that the derivatives 
satisfy
\[
\lambda_s(1,0) = -\int_{\Sigma_\Cc} \Psi_d \ d\mu_{\Cc}, \ \text{ and } \  \lambda_t(1,0) = \int_{\Sigma_\Cc} \Phi \ d\mu_{\Cc} = \int_{\Sigma_\Cc} \Psi_{d_\ast} - \Psi_d \ d\mu_{\Cc}
\]
where $\Cc$ is a fixed word maximal component and $\mu_\Cc$ is the equilibrium state for $-\Psi_d$ on $\Cc$. Further we know that
\[
1 = \tau(d_\ast/d) = \frac{ \int_{\Sigma_\Cc} \P_{d_\ast}  \ d\mu_\Cc}{\int_{\Sigma_\Cc} \P_{d} \ d\mu_\Cc}
\]
by Theorem 5.4 in \cite{cantrell.tanaka.2} and hence $\lambda_t(1,0) = 0$.
We also have  \cite{ParryPollicott}  that
\[
\lambda_{t t}(1,0) = \lim_{n\to\infty} \frac{1}{n} \int_{\Sigma_\Cc} (\Phi^n)^2 \ d\mu_{\Cc} \ge 0
\]
and that equality occurs if and only if $\Phi$ is cohomologous to a constant function on $\Sigma_\Cc$. If $\Phi$ is cohomologous to a constant function then this constant must be $0$ and so by Livsic's Theorem $\Phi^n(z) = 0$ for all $z \in \Sigma_\Cc$ with $\sigma^n(z) = z$. This happens if and only if $d$ and $d_\ast$ have the same translation length functions by \cite[Corollary 3.6]{cantrell} and this occurs if and only if they are roughly similar by Theorem \ref{thm.equiv}.
\end{proof}

It will transpire that the quantity $\sigma^2$ from Proposition \ref{prop.analytic} is given by
\begin{equation}\label{eq.variance}
\sigma^2(d_\ast/d) := - \frac{\lambda_{tt}(1,0)}{\lambda_s(1,0)} > 0.
\end{equation}
Before we prove Proposition \ref{prop.analytic} we need one last result regarding the spectral properties of the $L_{j,s,t}$. 
\begin{proposition}\label{prop.srless}
The assumption that either $\G$ is not virtually free or that $\G$ is virtually free and $d$ has non-arithmetic length spectrum guarantees the following property of the operators $L_{j,s,t}$ for $j =1, \ldots, m$. For each $s = 1 + il$ with $l \neq 0$ there exists $\epsilon(s) > 0$ such that each of the operators $L_{j,s,t}$ with $|t| < \epsilon(s)$ have spectral radius at most $1$ and furthermore do not have $1$ as an eigenvalue.
\end{proposition}
\begin{proof}
We recall that by Proposition \cite[Proposition 1.13]{cantrell} under the assumptions on $\G$, $d$ has non-arithmetic length spectrum. This guarantees that for each $s = 1 + il$ for $l\neq 0$ the operators $L_{j,s,0}$ do not have $1$ as an eigenvalue \cite[Proposition 4.4]{cantrell}. There are then two cases:\\
\textit{Case 1: $L_{j,s,0}$ has spectral radius strictly less than $1$.} In this case, by upper semi continuity of the spectrum there exists $\epsilon(s) > 0$ such that $L_{j,s,t}$ has spectral radius strictly less than $1$ when $s = 1 + il$ and $|t| < \epsilon(s)$.\\
\textit{Case 2: $L_{j,s,0}$ has spectral radius $1$.} In this case, by our above discussion $L_{j,s,0}$ has simple maximal eigenvalues of modulus $1$. Standard facts from analytic perturbation theory then guarantee that  there exists $\epsilon(s) >0$ such that $L_{j,s,t}$ has the simple maximal eigenvalues when $s = 1 + il$ and $|t| <\epsilon(s)$ and furthermore that they vary analytically as $t$ varies. In particular we can take $\epsilon(s)$ small so that $1$ is not an eigenvalue of $L_{j,s,t}$ for any $s = 1 + il, |t| < \epsilon(s)$. Lastly by Theorem 4.5 in \cite{ParryPollicott} the spectral radius of each $L_{j,s,t}$ is always at most $1$.
\end{proof}

We are now ready to prove Proposition \ref{prop.analytic}.
\begin{proof}[Proof of Proposition \ref{prop.analytic}]
We first show that each $\eta_p$ is analytic on the region $\mathfrak{Re}(s) \ge 1$ except for $s = 1$. By part (1) of Proposition \ref{prop.sr} we can write
\[
\eta(s,t) = \sum_{n \ge 1} \sum_{j=1}^k L_{j, s,t}^n \chi(\dot{0}) + \gamma(s,t)
\]
where $\gamma(s,t)$ is a function that is bi-analytic for $\mathfrak{Re}(s) > 1 - \epsilon$ for some $\epsilon >0$ and $|t|$ is sufficiently small depending on $s$. This is because the expression 
\[
\sum_{j=1}^k L_{j, s,t}^n \chi(\dot{0}) 
\]
over counts group elements corresponding to paths in the coding (starting at $\ast$) that do not enter a maximal component and by Proposition \ref{prop.mc} on these components the restricted pressure of $-\P_d$ is strictly less than $1$. We then note that for any fixed $\mathfrak{Re}(s) > 1$ and for $|t|$ sufficiently small depending on $s$ the spectral radius of any of the $ L_{j,s , t}$ is strictly less than $1$. Hence we can  differentiate $\eta(s,t)$ at $t =0$ for any $s$ with $\mathfrak{Re}(s) > 1$ to deduce analyticity of $\eta_p$ at $s$. When $s \neq 1$ has $\mathfrak{Re}(s) =1$ we can apply Proposition \ref{prop.srless} to deduce that each $L_{j, -s\Psi}$ either has spectral radius strictly less than $1$ or has $p_j$ (the period of $\Sigma_j$) simple maximal eigenvalues that does not include $1$. Hence for all such $s$ and for all $t$ sufficiently small $\eta(s,t)$ admits an analytic extension to $(s,0)$ which we can differentiate to give the required extension for $\eta_p$ at $s$.

We now consider the values of $s$ close to $1$. We know that there exists a complex neighbourhood $U \times V$ of $(1,0)$ such that for $(s,t) \in U \times V$ each of the operators $ L_{j, s ,t }$ have $p_j$ simple maximal eigenvalues of the form $e^{\pi i/p_j} e^{\Pr_j(-s\Psi_d - t\Phi)}$ where $p_j$ represents the period of the $j$th maximal component. The pressures $\Pr_j(-s\Psi_d - t\Phi)$ are bi-analytic in the neighbourhood $U \times V$. Further, by Lemma \ref{lem.pressagree} we know 
that $\Pr_j(-s\Psi_d - t\Phi)$ is independent of $j$ for all $(s,t) \in U \times V$. We write $\text{P}(s,t)$ for this function and $\lambda(s,t) = e^{\text{P}(s,t)}$ as before. It follows that for $(s,t) \in U\times V$ we can write 
\[
\eta(s,t) = \sum_{n \ge 1} \sum_{j=1}^k L_{j, s,t}^n \chi(\dot{0}) + \gamma(s,t) = \sum_{n\ge 1} \sum_{j=1}^k e^{n \Pr_j(-s\Psi_d - t\Phi)} Q_{j,s,t}\chi(\dot{0}) + \alpha(s,t)
\]
where $\alpha(s,t)$ is bi-analytic in a neighbourhood of $(1,0)$. Since the pressures coincide, setting $Q(s,t) = \sum_{j=1}^k Q_{j,s,t}\chi(\dot{0})$ we see that 
\begin{equation}\label{eq.1}
\eta(s,t) = \frac{Q(s,t) \lambda(s,t)}{1 - \lambda(s,t)} + \alpha(s,t).
\end{equation}
For later we note that $Q(1,0)$ is a positive real number. Indeed, by part (3) of Proposition \ref{prop.sr}, we can realise it as the limit
\[
Q(1,0) = \lim_{n\to\infty} \frac{1}{n} \sum_{i=1}^n \sum_{j=1}^k L_{j, 1,0}^i\chi(\dot{0}) = \lim_{n\to\infty} \frac{1}{n} \sum_{|g|_S \le n} e^{-d(o,g)}
\]
and there exists $C_1, C_2 >0$ such that
\[
C_1 \le \sum_{|g|_S = n} e^{-d(o,g)} \le C_2
\]
for all $n\ge 1$ by \cite[Lemma 2.8]{cantrell.tanaka.1}.

We now want to use expression (\ref{eq.1}) to deduce the required properties about the poles of $\eta_p$. 
We begin by considering the $p$th derivative of $(\ref{eq.1})$.  Writing $F(s,t) = Q(s,t) \lambda(s,t)$ we see that
\[
\left.\frac{\partial^p }{\partial t^p}\right|_{(s,0)} \eta(s,t) =  \left.\frac{\partial^p }{\partial t^p}\right|_{(s,0)} \left( \frac{F(s,t)}{1 - \lambda(s,t)} \right) + \beta(s)
\]
for some function $\beta$ that is analytic for $s$ near $1$. We then make the following claim.

\textit{Claim:} For each $p\ge 1$ there exists a function $R_p$ than is bi-analytic on a neighbourhood of $(1,0)$ such that if $p$ is even
\begin{equation}\label{eq.even}
\frac{\partial^p }{\partial t^p} \eta(s,t) = \frac{p!\lambda_{tt}(s,t)^{\frac{p}{2}}}{2^{\frac{p}{2}}(1-\lambda(s,t))^{1 + \frac{p}{2}}} F(s,t) + \frac{R_p(s,t)}{(1-\lambda(s,t))^{\frac{p}{2}}}
\end{equation}
and if $p$ is odd
\begin{equation}\label{eq.odd}
\frac{\partial^p }{\partial t^p} \eta(s,t) = \frac{(p-1)! (p+1) \lambda_{tt}(s,t)^{\frac{p-1}{2}}\lambda_t(s,t)}{2^{\frac{p+1}{2}}(1-\lambda(s,t))^{1 + \frac{p+1}{2}}} F(s,t) + \frac{R_p(s,t)}{(1-\lambda(s,t))^{\frac{p+1}{2}}}.
\end{equation}
Furthermore, for any $p\ge 1$ the following hold. If $p$ is even $R_p(s,0)(1-\lambda(s,0))^{-\frac{p}{2}}$  has poles of integer orders at most $p/2$ at $(1,0)$ and is analytic otherwise. If $p$ is odd then $R_p(s,0)(1-\lambda(s,0))^{-\frac{p+1}{2}}$ has integer poles of order at most $(p+1)/2$ at $(1,0)$ and is otherwise analytic.

This claim can be proved using Lemma \ref{lem.deriv} (specifically the fact that $\lambda_t(1,0) = 0$) and by induction. We leave the details to the reader.

We can now use the claim to conclude the proof. We  now distinguish the cases of $p$ being odd and even.\\
\textit{Case 1: $p$ is odd.} Evaluating expression $(\ref{eq.odd})$ at $(s,t) = (s,0)$ and using the fact that $\lambda_t(1,0) = 0$ we see that $\eta_p(s)$ has poles of integer order at most $(p+1)/2$. This is precisely what we want to show when $p$ is odd.\\
\textit{Case 2: $p$ is even.} Evaluating expression $(\ref{eq.even})$ at $(s,t) = (s,0)$ we see that 
\begin{equation}\label{eq.pderiv}
\left. \frac{\partial^p }{\partial t^p}\right|_{(s,0)} \eta(s,t) =  \frac{p! \lambda_{t t}(s,0)^{p/2}}{2^{p/2}(1-\lambda(s,0))^{1+p/2}}F(s,0) + R(s)
\end{equation}
where $R(s)$ is analytic in a neighbourhood of $s =1$ except for possible poles of orders at most $p/2$. Now recall that $F(s,0)$ is analytic in a neighbourhood of $s =1$ and also $F(1,0)$ is a strictly positive real number. Hence for $s$ near $1$, $F(s,0) = F(1,0) + (s-1)A(s)$ for some analytic $A$.  Similarly, by Lemma \ref{lem.deriv} we have that $\lambda_{ t t} (s,0) = \lambda_{t t}(1,0) + (s-1)B(s)$ for some analytic $B$ near $s =1$ and we have that $\lambda_{t t}(1,0) $ is a strictly positive real number. Lastly since $\lambda_s(1,0) < 0$ we have that $1 - \lambda(s,0) = (1-s) C(s)$ where $C$ is analytic in a neighbourhood of $1$ and 
$C(1) =  \lambda_s(1,0)$ is a strictly negative real number. Substituting the functions  $A, B, C$ into equation $(\ref{eq.pderiv})$ then shows that
\[
\left. \frac{\partial^p }{\partial t^p}\right|_{(s,0)} \eta(s,t) = \frac{F(1,0) p! \lambda_{t t}(1,0)^{p/2}}{2^{p/2} C(1)^{1+p/2} (1-s)^{1 + p/2}} L(s)
\]
where $L(s)$ is an analytic function on a neighbourhood of $s=1$ and $L(1) =1$. To conclude we note that
\[
 \frac{F(1,0) p! \lambda_{t t}(1,0)^{p/2}}{2^{p/2}C(1)^{1+p/2} (1-s)^{1 + p/2}} = \frac{F(1,0) p! \sigma^p}{2^{p/2}(-C(1))(s-1)^{1+p/2}}
\]
where $\sigma^2$ is as defined in (\ref{eq.variance}).
This concludes the proof and shows that the constants $C, \sigma^2$ in part $(2)$ of Proposition \ref{prop.analytic} are given by
\[
C = - \frac{F(1,0)}{\lambda_s(1,0)}\ \text{ and } \ \sigma^2 = \sigma^2(d_\ast/d) = - \frac{\lambda_{tt} (1,0)}{\lambda_s(1,0)},
\]
both of which are strictly positive.
\end{proof}


\section{Deducing the main theorems}\label{sec.proofs}
In this section we operate under the same assumptions that we stated at the beginning of Section \ref{sec.analytic}. As mentioned earlier, we  will use the method of moments. That is, we want to study the limit of the $p$th moments
\begin{equation} \label{eq.moments}
\lim_{T\to\infty} \frac{1}{\#\{ g\in\G: d(o,g) < T\}} \sum_{d(o,g) < T} \left(\frac{d_\ast(o,g) - T}{\sqrt{T}} \right)^p
\end{equation}
for each $p\ge 1$. 
We would like to show that for each $p\ge 1$ the moments written above in (\ref{eq.moments}) are equal to $0$ if $p$ is odd and to $\sigma^p (p-1)!!$ when $p$ is even \cite[Theorem 30.2]{billingsley}. Here, given $k \in \mathbb{Z}_{\ge 0}$, $k!!$ represents the product of the numbers between $1$ and $k$ that have the same parity as $k$. We need to be a little careful here: the method of moments applies to discrete sequence of random variables and our central limit theorem has convergence in a continuous variable $T$. We elaborate on this point during the proof of Theorem \ref{thm.main}.

We will actually study slightly different limits in $(\ref{eq.moments})$ in which $d_\ast(o,g) - T$ is replaced by $d_\ast(o,x) - d(o,x)$.
To extract the asymptotics we need from Proposition \ref{prop.analytic} we use the following Tauberian Theorem due to Delange \cite[Theorem III]{Delange}.
Throughout the rest of the paper, given two functions $f,g : \R \to \R$ we write $f(T) \sim g(T)$ if $f(T)/g(T) \to 1$ as $T\to\infty$.
\begin{proposition}\label{prop.tau}
For a monotone increasing function $\phi : \mathbb{R}_{>0} \to \mathbb{R}_{>0}$ we set
\[
f(s) = \int_0^{\infty} e^{-sT} \ d\phi(T).
\]
Suppose that there is $\delta > 0$ such that
\begin{enumerate}
\item $f(s)$ is analytic on $\{\mathfrak{Re}(s) \ge \delta\} \backslash \{\delta\}$; and,
\item  there are positive integers $n, k  \ge 1$, an open neighbourhood $U \ni \delta$, non-integer numbers $0 < \mu_1, \ldots, \mu_k < n$ and analytic maps $g,h, l_1, \ldots, l_k : U \to \mathbb{C}$ such that
\[
f(s) = \frac{g(s)}{(s-\delta)^n} + \sum_{j=1}^k \frac{l_j(s)}{(s-\delta)^{\mu_j}}+ h(s) \ \text{ for $s \in U$ and such that $g(\delta) > 0$}.
\]
\end{enumerate}
Then
\[
\phi(T) \sim \frac{g(\delta)}{(n-1)!} T^{n-1} e^{\delta T}
\]
as $T\to\infty$.
\end{proposition}
It will be useful to make the following definition.
\begin{definition}
For each $p \ge 1$ we set
\[
\pi_p(T) = \sum_{d(o,g) < T} (d_\ast(o,g) -  d(o,g))^p
\]
for $T >0$. 
\end{definition}
Using $\pi_p$ we can write
\[
\sum_{g\in\G} (d_\ast(o,g) -  d(o,g))^p e^{-sd(o,g)} = \int_0^\infty e^{-sT} \ d\pi_p(T).
\]
When $p$ is even Proposition \ref{prop.tau} applies (as $\pi_p$ is monotone increasing) and we see that 
\[
\pi_p(T) \sim \frac{Cp!\sigma^p}{2^{p/2}(\frac{p}{2})!} e^{T} T^{\frac{p}{2}}= C\sigma^p(p-1)!! \,e^{T} T^{\frac{p}{2}}  \ \text{ and } \ \pi_0(T) \sim Ce^{T}  
\]
as $T\to\infty$. Here $C,\sigma^2 > 0$ are as in part $(2)$ of Proposition \ref{prop.analytic}. In now follows that
\[
\lim_{T\to\infty} \frac{1}{\#\{ g\in\G: d(o,g) < T\}} \sum_{d(o,g) < T} \left(\frac{d_\ast(o,g) -  d(o,g)}{\sqrt{T}} \right)^p = \lim_{T\to\infty} \frac{\pi_p(T)}{T^{\frac{p}{2}} \pi_0(T)} = \sigma^p (p-1)!! .
\]
This is what we want to show in the case that $p$ is even. 

Before continuing with the proof of Theorem \ref{thm.main} we  prove Theorem \ref{thm.var} and Theorem \ref{thm.var2}.
\begin{proof}[Proof of Theorem \ref{thm.var} and Theorem \ref{thm.var2}]
Recall that we are assuming that $d,d_\ast$ are not roughly similar. Setting $p=2$ the previous display equation gives that
\begin{equation}\label{eq.m1}
\lim_{T\to\infty} \frac{1}{\#\{ g\in\G: d(o,g) < T\}} \sum_{d(o,g) < T} \left(\frac{d_\ast(o,g) -  d(o,g)}{\sqrt{T}} \right)^2 = \sigma^2(d_\ast/d).
\end{equation}
The limit in Theorem \ref{thm.var} is similar except that $d_\ast(o,x) - d(o,x)$ is replaced by $d_\ast(o,x) - T$. By splitting the sum over $\{ x \in\G: d(o,x) <T\}$  in $(\ref{eq.m1})$ into sums over the sets $\{ x\in\G: d(o,x) \le T-T^{1/3}\}$ and $\{ x\in\G: T-T^{1/3} < d(o,x) < T\}$ it is easy to deduce the required limit. This uses the fact that 
\[
\frac{\#\{ x\in\G: d(o,x) \le T-T^{1/3}\}}{\#\{x \in \G : d(o,x) < T\}} = O\left( e^{-T^{1/3}}\right)
\]
 decays to zero faster than any polynomial as $T\to\infty$. We leave the details to the reader.

We also need to identify $\sigma^2(d_\ast/d)$ with the second derivative $\theta_{d_\ast/d}''(0)$. To do this let $w(t)$ be the implicit solution to $\lambda(w(t) ,t) = 0$ which exists and is analytic for $t$ in a neighbourhood of $0$ by the Implicit Function Theorem (which we can apply due to Lemma \ref{lem.deriv}). The Implicit Function Theorem also tells us that
\begin{equation}\label{eq.wderiv}
w'(t) = - \frac{\lambda_t(w(t), t)}{\lambda_s(w(t),t)}
\end{equation}
from which it follows that $w'(0) = 0$.
Now note that 
\[
P(w(t), t) = \Pr_\Cc( (t-w(t)) \P_d  - t\P_{d_\ast} )
\]
where $\Cc$ is any word maximal component by Theorem 5.4 of \cite{cantrell.tanaka.2} and Proposition \ref{prop.mc}. Hence when $t \in \R$ we see that for $\theta_{d_\ast/d}(t) = w(t) - t$ and in particular $\theta''_{d_\ast/d}(0) = w''(0)$. Hence to conclude the proof it suffices to show that $w''(0) = \sigma^2(d_\ast/d)$. We now differentiate $(\ref{eq.wderiv})$ which shows that $ - w''(t)$ is given by
\[
\frac{(\lambda_{t t}(w(t), t) + \lambda_{s t}(w(t),t)w'(t)) \lambda_s(w(t),t) - (\lambda_{t s}(w(t), t) + \lambda_{s s}(w(t), t) w'(t)) \lambda_t(w(t),t)}{(\lambda_s(w(t),t))^2}.
\]
Evaluating at $t = 0$ and using the fact that $w'(0) = \lambda_t(1,0) = 0 $ we see that
\[
w''(0) = - \frac{\lambda_{t t}(1,0)}{\lambda_s(1,0)}
\]
which agrees with $\sigma^2(d_\ast/d)$ due to  $(\ref{eq.variance})$. To finish the proof we need to discuss the case that $d,d_\ast$ are roughly similar. When this is the case it is easy to see that $\sigma^2(d_\ast/d) = 0$ by Theorem \ref{thm.equiv}. The equivalence statement in Theorem \ref{thm.var} then follows and the proof is complete.
\end{proof}

We now continue with the proof of Theorem \ref{thm.main} by studying the moments when $p$ is odd. To prove the following proposition we follow an argument due to Hwang and Janson \cite{hwang.janson}.
\begin{proposition}
When $p$ is odd we have that
\[
\lim_{T\to\infty} \frac{1}{\#\{ g\in\G: d(o,g) < T\}} \sum_{d(o,g) < T} \left(\frac{d_\ast(o,g) - d(o,g)}{\sqrt{T}} \right)^p =   0.
\]
\end{proposition}

\begin{proof}
Fix odd $p \ge 1$.
We define
\begin{align*}
&G_1(s) = \sum_{g\in\G} \left( (d_\ast(o,g) - d(o,g))^{2p} + d(o,g)^p)\right) e^{-sd(o,g)}\\
&G_2(s) = \sum_{g\in\G} \left((d_\ast(o,g) - d(o,g))^p + d(o,g)^{\frac{p}{2}} \right)^2 e^{-sd(o,g)}\\
&G_3(s) = \sum_{g\in\G} d(o,g)^{\frac{p}{2}} (d_\ast(o,g) - d(o,g))^p e^{-sd(o,g)}.
\end{align*}
Note that $G_2 = G_1 + 2G_3$ and that these three series are convergent in $\mathfrak{Re}(s) > 1$.
We also have that
\[
G_1(s) = \eta_{2p}(s) - \eta_0^{(p)}(s) = \frac{g(s)}{(s-1)^{1+p}} + f(s)
\]
where $g(s), f(s)$ are analytic in $\mathfrak{Re}(s) \ge 1$ and $g(1)$ is a positive real number. Here the power of $p$ in parenthesise (in $\eta_0^{(p)}$) represents taking the $p$th derivative. We will use this notation throughout the rest of the paper.
It follows from Proposition \ref{prop.tau} that
\[
\frac{1}{\#\{g\in\G: d(o,g) <T\}} \sum_{d(o,g) <T} (d_\ast(o,g) - d(o,g))^{2p} + d(o,g)^p  \sim \frac{g(1)T^p}{p!}
\]
as $T\to\infty$. 

We now want to show that $G_2$ has a pole with the order poles at $s=1$. We begin by calculating
\[
\eta_p^{\left( \frac{p+1}{2}\right)}(s) = (-1)^{\left( \frac{p+1}{2}\right)} \sum_{d(o,g) <T} d(o,g)^{ \frac{p+1}{2}}(d_\ast(o,g) - d(o,g))^p e^{-sd(o,g)}.
\]
Then using the identity
\[
\int_0^\infty t^{-\frac{1}{2}} e^{-tx} \ dt = \frac{\sqrt{\pi}}{\sqrt{x}}
\]
for $x >0$ we can write
\begin{align*}
G_3(s) &=  \frac{1}{\sqrt{\pi}} \sum_{g\in\G} d(o,g)^{\frac{p+1}{2}} (d_\ast(o,g) - d(o,g))^p e^{-sd(o,g)} \int_0^\infty t^{-\frac{1}{2}} e^{-td(o,g)} \ dt\\
&= \frac{(-1)^{\frac{p+1}{2}}}{\sqrt{\pi}} \int_0^\infty \frac{\eta_p^{\left( \frac{p+1}{2}\right)}(s+t) }{\sqrt{t}} \ dt.
\end{align*}
It follows that $G_3$ is analytic on $\mathfrak{Re}(s) \ge 1$ except for a pole at $s=1$. We also know that $\eta_p^{\left( \frac{p+1}{2}\right)}$ has a pole of order at most $p+1$ at $s=1$. It follows that locally to $s=1$ we can write
\[
\eta_p^{\left( \frac{p+1}{2}\right)}(s) = \sum_{j=1}^{p+1} \frac{a_j}{(s-1)^j} + h(s)
\]
where $h(s)$ is analytic and the $a_1,\ldots, a_{p+1}$ are complex numbers (some of which could be $0$).
Since
\[
\int_0^\infty \frac{1}{(s+t -1)^j \sqrt{t}} \ dt = \frac{\pi (2j-2)!}{2^{2j-1}((j-1)!)^2} \, \frac{1}{(s-1)^{k-\frac{1}{2}}}
\]
it follows that in a neighbourhood of $s=1$
\[
G_3(s) = \sum_{j=1}^{p+1} \frac{c_j}{(s-1)^{j-\frac{1}{2}}} + l(s)
\]
for some analytic function $l$ and constants $c_1, \ldots, c_{p+1}$. Then using the identity $G_2 = G_1 + 2G_3$  we see that
\[
G_2(s) = \sum_{j=1}^{p+1} \frac{a_j}{(s-1)^j} + \frac{b_j}{(s-1)^{j-\frac{1}{2}}} + r(s)
\]
where $a_j, b_j,r$ are analytic maps on $\mathfrak{Re}(s) \ge 1$. Furthermore $a_{p+1}(1) = g(1)$ (where $g$ is the function from our expression for $G_1$). We can then apply Proposition \ref{prop.tau} to $G_2$ to deduce that
\[
\frac{1}{\#\{g\in\G: d(o,g) <T\}} \sum_{d(o,g) < T} \left((d_\ast(o,g) - d(o,g))^p + d(o,g)^{\frac{p}{2}} \right)^2 \sim \frac{g(1)T^{p}}{p!}
\]
as $T \to\infty$. It then follows that
\[
\frac{1}{\#\{g\in\G:d(o,g) <T\}}\sum_{d(o,g) < T} d(o,g)^{\frac{p}{2}} (d_\ast(o,g) - d(o,g))^p = o(T^p)
\]
as $T\to\infty$.
To conclude the proof we need to show that the same estimate holds when we replace $d(o,g)^{\frac{p}{2}}$ with $T^{\frac{p}{2}}$, i.e. we want to show that $\pi_p(T) = o(T^{\frac{p}{2}} e^T)$.
To do so, note that
\begin{equation}\label{eq.pi}
\sum_{d(o,g) < T} d(o,g)^{\frac{p}{2}} (d_\ast(o,g) - d(o,g))^p = \int_0^T t^{\frac{p}{2}} \ d\pi_p(t) = T^{\frac{p}{2}} \pi_p(T) - \frac{p}{2} \int_0^T t^{\frac{p}{2} -1} \pi_p(t) \ dt.
\end{equation}
This follows from a standard integration by-parts formula for the Stiltjes integral.
Now using the fact that $|\pi_p(T)| \le \pi_{p-1}(T) +\pi_{p+1}(T)  = O(T^{\frac{p+1}{2}} e^T)$ (where we have used that $p-1$ and $p+1$ are even and so our estimates on the even moments apply) it is a simple calculation to show that 
\[
\left| \int_0^T  t^{\frac{p}{2} -1} \pi_p(t) \ dt \right| =O\left( \int_0^T t^{p -\frac{1}{2}} e^t \ dt\right)  = o(T^{p} e^T)
\]
as $T \to\infty$.
Substituting this into $(\ref{eq.pi})$ shows that
\[
 \pi_p(T) = T^{-\frac{p}{2}}\left( \sum_{d(o,g) < T} d(o,g)^{\frac{p}{2}} (d_\ast(o,g) - d(o,g))^p  + \frac{p}{2} \int_0^T t^{\frac{p}{2} -1} \pi_p(t) \ dt
\right) = o(T^p e^T)
\]
as $T\to\infty$, as required.
\end{proof}

We have shown the following.
\[
\lim_{T\to\infty} \frac{1}{\#\{ g\in\G: d(o,g) < T\}} \sum_{d(o,g) < T} \left(\frac{d_\ast(o,g) -  d(o,g)}{\sqrt{T}} \right)^p =     \begin{cases}
       0 &\quad \text{if $p$ is odd} \\
       \sigma^p (p-1)!! &\quad \text{if $p$ is even.}
     \end{cases}
\]
We are now ready to conclude the proof of Theorem \ref{thm.main}.
\begin{proof}[Proof of Theorem \ref{thm.main}]
As mentioned, we would like to apply the method of moments \cite[Theorem 30.2]{billingsley}. 
Considering the moments above along the integers $n\ge 1$ we deduce that for any $t\in\R$
\[
\frac{1}{\#\{ g \in \G: d(o,g) < n\}} \#\left\{ g \in \G: d(o,g) < n\ \text{ and } \ \frac{d_\ast(o,g) - d(o,g)}{\sqrt{n}} \le t  \right\} \to \mathcal{N}_\sigma(t)
\]
as $n\to\infty$ (in the discrete sense). Now note that there exists $M>0$ such that when $T \in [n, n+1)$ and $d(o,x) < T$, $|d_\ast(o,x) - d(o,x)| \le M(n + 1)$ and so  for each $n \ge 1$ and so there exists $R >0$ such that when $T \in [n,n+1)$
\[
\left| \frac{d_\ast(o,x) - d(o,x)}{\sqrt{T}} - \frac{d_\ast(o,x) - d(o,x)}{\sqrt{n}}\right| \le \frac{R}{\sqrt{n}}
\]
for all $n\ge 1$. 
Writing
\[
\mathrm{C}(T,t) = \frac{1}{\#\{ g \in \G: d(o,g) < T\}} \#\left\{ g \in \G: d(o,g) < T \ \text{ and } \ \frac{d_\ast(o,g) - d(o,g)}{\sqrt{T}} \le t  \right\}
\]
we therefore have that $\mathrm{C}(n,t) \le \mathrm{C}(T,t) \le \mathrm{C}(n, t + Rn^{-1/2})$ for each $n\ge 1$, $T \in [n,n+1)$ and $t \in \R$. It follows that for each $t \in \R$, $\mathrm{C}(T,t)$ tends to $\mathcal{N}_\sigma(t)$ as $T\to\infty$ (in the continuous sense).

This is almost the statement of the central limit theorem in Theorem \ref{thm.main} except we would like  $d_\ast(o,x) - d(o,x)$ replaced with $d_\ast(o,x) - T$. However, this follows easily as we have that
\[
\#\{ x\in \G: d(o,x) < T  \text{ and } |d(o,x) - T| > T^{1/3}\} = o(\#\{ x \in \G: d(o,x) < T\})
\]
as $T\to\infty$.
\end{proof}
\section{Further applications and examples}
In this last section we include some further applications of our results and methods. In particular we look at central limit theorems for real valued potentials (other than strongly hyperbolic metrics) on $\G$ when we order elements by their length with respect to a strongly hyperbolic metric. They key observation is that our arguments can be applied to potentials $\varphi : \G \to \R$ that satisfies the following properties:
\begin{enumerate}[(i)]
\item there exists a finite, symmetric generating set $S$ for $\G$ and constant $C >0$ such that for any $g \in \G$ and $s \in S$ both $|\varphi(gs) - \varphi(g)| < C$ and $|\varphi(sg) - \varphi(g)| < C$; and,
\item there exists a Cannon coding (for some generating set) with corresponding shift space $\Sigma$ and a potential $\P_{\varphi}$ such that for $x \in \Sigma$ with $x_0 = \ast$, $S_n \P_{\varphi}(x)= \varphi(\ev_n(x))$.
\end{enumerate}
Once we have a function $\varphi: \G \to \mathbb{R}$ that satisfies these conditions, our arguments apply and we can deduce a central limit theorem for $\varphi$. This can be easily be verified from following our proof of Theorem \ref{thm.main}. The non-degeneracy criteria then becomes the following. The central limit theorem comparing $\varphi$ and a strongly hyperbolic metric $d$ is degenerate if and only if there exists $C >0, \tau \in \R$ such that $|\varphi(g) - \tau d(o,g) |<C$ for all $g \in \G$. We now discuss various examples of $\varphi$ satisfying the above conditions.

\subsection{Word metrics and bi-combable functions}
As mentioned in the introduction, there has been much interest in counting central limit theorems when group elements are ordered by a word metric. Our methods allow us to prove central limit theorems for word metrics but when the ordering is done with respect to a different metric.

\begin{theorem}
Let $\G$ be non-elementary, non-virtually free hyperbolic group and $d$ a strongly hyperbolic metric in $\Dc_\G$. Fix a word metric $|\cdot|_S$ on $\G$ associated to a finite generating set $S$. Then there exists $\tau > v_d/v_S$ and $\sigma^2 > 0$ such that for any $t \in \R$,
\[
\frac{1}{\#\{ g\in\G: d(o,g) < T\}} \#\left\{ g\in\G: d(o,g) < T \ \text{ and } \ \frac{|g|_S - \tau T}{\sqrt{T}}  \le t \right\} \to \mathcal{N}_\sigma(t)
\]
as $T\to\infty$.
\end{theorem}
Here $v_S$ denotes the exponential growth rate of the word distance $d_S$ associated to $S$. Note that here we know that $\sigma^2$, the variance, is strictly positive.
\begin{proof}
We use a Cannon coding for the word metric $S$ in the theorem. Then, the constant potential with value $1$ on this coding represents the word metric $|\cdot|_S$ in the sense of $(ii)$ above.
It is also clear that $(i)$ is satisfied for the map $g \mapsto |g|_S$.  We therefore have the required properties needed to follow the same proof used to show Theorem \ref{thm.main}. To conclude the proof we need to know that $d$ and $d_S$ are never roughly similar to guarantee that $\tau > v_d/v_S$ and that $\sigma^2 > 0$. This follows from and argument of Gou\"ezel, Matheus and Maucorant \cite{GMM2018} which shows that $d$ can not have arithmetic length spectrum (where as it is known that $d_S$ necessarily does have arithmetic length spectrum). See \cite[Proposition 1.13]{cantrell} for further explanation.
\end{proof}
In fact we can apply our result to the class of bi-combable functions introduced by Calegari and Fujiwara in \cite{CalegariFujiwara2010} (which include word metrics). These functions satisfy $(i)$ and $(ii)$ above (by definition). An interesting example of bi-combable functions are Brook's counting quasi-morphisms. Our work applies and we deduce counting central limit theorems comparing bi-combable functions with strongly hyperbolic metrics. Such central limit theorems are always non-degenerate.

\subsection{Quasi-Fuchsian representations}
For our final example we compare two distances on a surface group $\G$, i.e. $\G$ is the fundamental group of a closed, negatively curved Riemannian surface $(V,\mathfrak{g})$. Given such a group $\G$ we can lift the metric $\mathfrak{g}$ on the surface to obtain a metric $d_\mathfrak{g}$ on $\G$. More precisely, if $(\widetilde{V}, \widetilde{\mathfrak{g}})$ is the universal cover of $(V,\mathfrak{g})$ and we fix a base point $o\in\widetilde{V}$ then we can define
\[
d_\mathfrak{g}(g,h) = d_{\widetilde{\mathfrak{g}}}(g \cdot o, h \cdot o) \ \text{ for $g,h \in \G$.}
\]
 We can also take a representation $\rho : \G \to \PSL_2(\C) \simeq \text{Isom}(\H^3)$ corresponding to convex cocompact isometric action by $\G$ on $\H^3$. From such a representation we obtain a metric $d_{\rho}$  on $\G$ by lifting as we did for $d_\mathfrak{g}$. We then have the following.

\begin{theorem}
Suppose that $\G$ is a surface group and that $d_\mathfrak{g}$ and $d_\rho$ are as above. Then, there exists $\tau > 0$, $\sigma^2 \ge 0$ such that
\[
\frac{1}{\#\{ g\in\G: d_\mathfrak{g}(o,g) < T\}} \#\left\{ g\in\G: d_\mathfrak{g}(o,g) < T \ \text{ and } \ \frac{d_{\rho}(o,g) - \tau T}{\sqrt{T}} \le t  \right\} \to \mathcal{N}_\sigma(t)
\]
as $T\to\infty$. Furthermore, $\sigma^2 = 0$ if and only if $\mathfrak{g}$ corresponds  to a point in Teichm\"uller  space $\rho_\mathfrak{g} : \G \to \PSL_2(\R)$ (i.e. $\mathfrak{g}$ has constant negative curvature) and $\rho$  is a conjugate representations in $\PSL_2(\C)$, i.e. there exists $M \in \PSL_2(\C)$ such that $\rho(g) = M^{-1}\rho_\mathfrak{g}(g) M$ for all $g \in \G$.
\end{theorem}

\begin{proof}
Both metrics $d_\rho$ and $d_\mathfrak{g}$ are strongly hyperbolic and so Theorem \ref{thm.main} applies and the central limit theorem follows. The non-degeneracy equivalence statement follows from a result of Fricker and Furman \cite[Theorem A]{fricker.furman}.
\end{proof}

In the above theorem we could swap the roles of $d_\mathfrak{g}$ and $d_\rho$ to obtain a central limit theorem with possibly different mean and variance.
\subsection{Translation distance functions}
To finish this work we briefly discuss versions of Theorem \ref{thm.main} for translation distance functions. We begin with the following lemma.

\begin{lemma}
Let $d \in \Dc_\G$ be strongly hyperbolic. Then we have that
\[
\frac{1}{\#\{g \in \G: d(o,g) < T\}} \#\{x\in\G: d(o,g) < T, \text{ and } |d(o,g) - \ell_d(g)| > T^{1/3} \} \to 0
\]
as $T\to\infty$.
\end{lemma}
\begin{proof}
Consider the set of elements
\[
\mathcal{U}(T) = \{g\in\G: d(o,g) < T, \text{ and } |d(o,g) - \ell_d(g)| > T^{1/3} \}.
\]
Now there exists $C>0$ such that given $x \in \mathcal{U}(T)$ we can find, by \cite[Lemma 3.2]{cantrell.tanaka.1}, $g' \in [g]$ such that $|d(o,g') - \ell_d(g')| = |d(o,g') - \ell_d(g)| \le C$. Then we have that $d(o,g') \le \ell_d(g) + C \le T - T^{1/3}$. Hyperbolicity guarantees that the association $g \mapsto g'$ is at most $K (d(o,g) + 1) $ to $1$ for some $K > 0$ independent of $g$ (i.e it is at most $K(T+1)$ to $1$ for $g\in \mathcal{U}(T)$). Hence we have that
\[
\#\mathcal{U}(T) \le K \, (T+1)\#\{g' \in \G: d(o,g') < T- T^{1/3} \} = O\left( T e^{v_d(T-T^{1/3})}\right)
\]
as $T\to\infty$  and the result follows.
\end{proof}
As a corollary of this lemma and Theorem \ref{thm.main} we deduce the following result. Recall that, given $d_\ast \in \Dc_\G$ and $g \in \G$, $\ell_{d_\ast}[g]$ is stable translation length of (the conjugacy class containing) $g$ with respect to $d_\ast$.
\begin{corollary}
Suppose that $\G$ is a non-elementary hyperbolic group and take two strongly hyperbolic metrics $d,d_\ast$ on $\G$. Suppose either that $\G$ is not virtually free or that $\G$ is virtually free and that $d$ has non-arithmetic length spectrum. Then there exist $\tau = \tau(d_\ast/d) >0$ and $\sigma^2 = \sigma^2(d_\ast/d) \ge 0$ such that
\[
\frac{1}{\#\{ g \in \G: d(o,g) < T\}} \#\left\{ g \in \G: d(o,g) < T \ \text{ and } \ \frac{\ell_{d_\ast}[g] - \tau T}{\sqrt{T}} \le t  \right\} \to \mathcal{N}_\sigma(t)
\]
as $T\to\infty$. 
Furthermore, $\sigma^2(d_\ast/d) = 0$ if and only if $d$ and $d_\ast$ are roughly similar.
\end{corollary}
Given this result it is natural to ask whether we can also replace $d$ with $\ell_d$ and count over conjugacy classes opposed to group elements. We are unsure how to approach this problem and so end the paper with the following question.
\begin{question}
Let $d,d_\ast$ be strongly hyperbolic metrics on a non-elementary hyperbolic group $\G$. Does a central limit theorem hold that compares $\ell_d$ and $\ell_{d_\ast}$ on conjugacy classes. That is, does
\[
\frac{1}{\#\{ [g] \in \conj(\G): \ell_d[g] < T\}} \#\left\{ [g] \in \conj(\G): \ell_d[g] < T \ \text{ and } \ \frac{\ell_{d_\ast}[g] - \tau T}{\sqrt{T}} \le t  \right\} 
\]
converge to a normal distribution as $T\to\infty$?
\end{question}
\subsection*{Open access statement}
For the purpose of open access, the authors have applied a Creative Commons Attribution (CC BY) licence to any Author Accepted Manuscript version arising from this submission.

\bibliographystyle{alpha}
\bibliography{coding}

\end{document}